\documentclass[a4paper]{amsart}
\numberwithin{equation}{section}
\newtheorem{theorem}{Theorem}[section]
\newtheorem{corollary}{Corollary}[section]
\newtheorem{definition}{Definition}[section]
\newtheorem{lemma}{Lemma}[section]

\theoremstyle{remark}
\newtheorem{remark}{Remark}[section]
\newtheorem{conjecture}{Conjecture}[section]
\usepackage{hyperref}
\usepackage{color}
\usepackage{graphics}
\usepackage{graphicx}
\usepackage{epsf}
\usepackage{amsfonts}
\usepackage{subfigure}
\usepackage{geometry}
\usepackage{tikz}
\title[Strongly starlike functions]
 {On a certain subclass of strongly starlike functions}
\subjclass[2020]{30C45; 30C50}
\keywords{analytic; univalent; strongly starlike functions; subordination; coefficient estimates.}
\begin{document}
\begin{abstract}
Let $\mathcal{S}^*(\alpha_1,\alpha_2)$, where $ \alpha_1, \alpha_2 \in (0,1]$, represent the class of functions $f$ that are analytic in the open unit disk $\mathbb{D}$, normalized by $f(0) = f'(0) - 1=0$, and satisfying the following double-sided inequality:
\begin{equation*}
    -\frac{\pi\alpha_1}{2}< \arg\left\{\frac{zf'(z)}{f(z)}\right\}
<\frac{\pi\alpha_2}{2}, \quad (z\in\mathbb{D}).
\end{equation*} 
In this manuscript, we estimate the coefficients and logarithmic coefficients associated with functions that belong to the class $\mathcal{S}^*(\alpha_1,\alpha_2)$. As a result, we provide a general bound for the coefficients of a strongly starlike function, which has been an open question until now. Finally, we derive upper and lower bounds for the expression ${\rm Re}\{zf'(z)/f(z)\}$, where $f\in \mathcal{S}^*(\alpha_1,\alpha_2)$.
\end{abstract}
\author[R. Kargar, J. Sok\'{o}{\l} and H. Mahzoon]
       {Rahim Kargar, Janusz Sok\'{o}{\l} and Hesam Mahzoon}
\address{Department of Mathematics and Statistics, University of Turku, Turku, Finland}
       \email {rahim.r.kargar@utu.fi {\rm (R. Kargar)}}
\address{University of Rzesz\'{o}w, Faculty of Exact and Technical Sciences, ul. Prof. Pigonia 1, 35-310 Rzesz\'{o}w, Poland}
       \email{jsokol@ur.edu.pl {\rm (J. Sok\'{o}{\l})}}
\address{Department of Mathematics, West Tehran Branch, Islamic Azad University, Tehran, Iran}
\email {hesammahzoon1@gmail.com {\rm (H. Mahzoon)}}

\maketitle
\section{Introduction}\label{sec1}

Univalent functions, also known as one-to-one functions or injective functions, play a fundamental role in complex analysis and in various branches of mathematics. These functions possess a remarkable property: they map distinct elements from their domain to distinct elements in their co-domain. In other words, they establish a unique correspondence between input values and output values, avoiding any duplication or overlap. Univalent functions have wide-ranging applications in fields such as geometry, conformal mapping, and the theory of Riemann surfaces. Their study not only deepens our understanding of mathematical structures but also finds practical applications in various areas of science and engineering. In this exploration of univalent functions, we investigate a subclass of strongly starlike functions, which is a certain subclass of univalent functions. Strongly starlike functions are a class of univalent complex-valued functions in the field of complex analysis. These functions are defined with respect to the unit disk $\mathbb{D}=\{z\in \mathbb{C}: |z|<1\}$. 
In this context, let us remember that a holomorphic and univalent function $f(z)$, which is normalized and defined within the unit disk $\mathbb{D}$, is classified as a strongly starlike function if and only if, for any point $w$ situated outside the image of $f(\mathbb{D})$ in the complex plane $\mathbb{C}$, we can establish that $w$ serves as the vertex of an angular sector. This sector has an opening with a measure of $(1-\beta)\pi/2$, where $0 < \beta \leq 1$, and is completely contained within the region outside $f(\mathbb{D})$. Furthermore, this angular sector is precisely bisected by the radial vector originating from the point $w$, see \cite{Stan2}.

On the other hand, strongly starlike functions are a subset of starlike functions, where the condition is stronger, ensuring a more restricted behavior of the image of the unit disk under the function. These functions are of interest in complex analysis and have applications in various areas, including the study of univalent functions and geometric function theory.

Recently, there has been a surge of interest among researchers in studying the class of strongly starlike functions. See, for example, \cite{AP-2017, ADS, KA, NS, PS97}. This newfound attention reflects growing curiosity about the unique properties and behavior of these functions within the field of complex analysis. Researchers have been actively exploring the intricate characteristics and mathematical properties of strongly starlike functions to deepen our understanding of their role in various mathematical contexts.

In addition to the broader exploration of strongly starlike functions, our research will also focus on a particular subclass within this category. We are motivated to explore these specific subsets' intricate details and unique features. By focusing on this subclass, we seek to better understand the behaviors and properties that distinguish these functions from the broader class of strongly starlike functions. This focused approach will allow us to uncover unique insights and new mathematical results within this specialized domain.

The organization of the paper can be summarized as follows: Section \ref{sec2} introduces initial definitions and several lemmas, while Section \ref{sec-Main Results} presents our central findings.

\section{Preliminaries}\label{sec2}
Consider the set $\mathcal{H}$ consisting of functions $f$ that are holomorphic in the open unit disk $\mathbb{D}=\{z\in \mathbb{C} : |z|<1\}$. Furthermore, let $\mathcal{A}$ be a subset of $\mathcal{H}$ consisting of functions $f$ that can be represented in the form:
\begin{equation}\label{f}
f(z)=z+\sum_{n=2}^{\infty}a_{n}z^{n},\quad(z\in\mathbb{D}).
\end{equation}
These functions in $\mathcal{A}$ are subject to a normalization condition, namely, $f(0)=0$ and $f'(0)=1$, when evaluated within the unit disk $\mathbb{D}$.
The subclass of $\mathcal{A}$ consisting of
all univalent functions $f$ in $\mathbb{D}$ will be denoted by
$\mathcal{S}$. We say that a function $f\in \mathcal{S}$ is starlike if and only if
\begin{equation*}
  {\rm Re}\left\{\frac{zf'(z)}{f(z)}\right\}>0,\quad (z\in \mathbb{D}).
\end{equation*}
We denote by $\mathcal{S}^*$ the class of starlike functions. Also, we say that a function $f\in \mathcal{S}$ is strongly starlike of order $\beta$, where $0<\beta\leq 1$ if, and only if,
\begin{equation*}
  \left|\arg\left\{\frac{zf'(z)}{f(z)}\right\}\right|<\frac{\pi \beta}{2},\quad (z\in \mathbb{D}).
\end{equation*}
The class of strongly starlike functions of order $\beta$ is denoted by $\mathcal{SS}^*(\beta)$. The class $\mathcal{SS}^*(\beta)$ was introduced independently by Stankiewicz (see \cite{Stan1}, \cite{Stan2}) and by Brannan and Kirvan (see \cite{Brannan}).
We remark that $\mathcal{SS}^*(1)\equiv \mathcal{S}^*$.

In the open unit disk $\mathbb{D}$, two analytic functions, denoted as $f$ and $g$, are considered in a subordinate relationship if $f(z)\prec g(z)$ or simply $f\prec g$. This relationship holds when there exists a Schwarz  function $w$ defined in $\mathbb{D}$, satisfying the initial conditions $w(0)=0$ and $|w(z)|\leq 1$, and such that
$f(z)=g(w(z))$ for all $z\in\mathbb{D}$.

In this paper, we consider the following analytic function
\begin{equation}\label{g}
    G(z):=G(\alpha_1,\alpha_2,c)(z)
    =\left(\frac{1+cz}{1-z}\right)^{(\alpha_1+\alpha_2)/2},\quad
    (z\in \mathbb{D}),
\end{equation}
where $0<\alpha_1, \alpha_2\leq 1$, $c=e^{\pi i\theta}$, and
$\theta=(\alpha_2-\alpha_1)/(
\alpha_2+\alpha_1)$. The function $G$ is analytic and convex univalent in the unit disk $\mathbb{D}$ and maps $\mathbb{D}$ conformally onto 
\begin{equation*}
    \left\{w\in\mathbb{C}: -\frac{\pi\alpha_1 }{2}<\arg\, w < \frac{\pi\alpha_2}{2}\right\},
\end{equation*}
see \cite[Lemma 2]{Yang} for more details.
It is also worth mentioning that this function is a special case of the function defined in \cite[(2.2)]{AP-2017}, where $x=c$ and
\begin{equation}\label{alpha}
\alpha=\frac{\alpha_1+\alpha_2}{2}.
\end{equation}
It follows from \cite[(2.2)]{AP-2017} that
\begin{equation*}
G(z)=1+\sum_{n=1}^{\infty}A_n z^n,
\end{equation*}
where
\begin{equation}\label{coef-An}
A_n
:=A_n(\alpha_1,\alpha_2,c)=
\sum_{k=1}^{n}
\binom{n-1}{k-1}
\binom{(\alpha_1+\alpha_2)/2}{k}
(1+c)^k,\quad
( n\ge1).
\end{equation}
It is easy to see that
\begin{equation}\label{A123}
  A_n=\left\{\begin{array}{ll}
\alpha(1+c), & \quad n=1;\\\\
 \alpha(1+c)+\frac{\alpha(\alpha-1)}{2}(1+c)^2, & \quad n=2;\\\\
 \alpha(1+c)+\alpha(\alpha-1)(1+c)^2
+\frac{\alpha(\alpha-1)(\alpha-2)}{6}(1+c)^3, &\quad n=3,\\
\end{array} \right.
\end{equation}
where $\alpha$ is given by \eqref{alpha}.
Moreover, by  \cite[(2.3)]{AP-2017} we have
\begin{equation*}
A_n(\alpha_1,\alpha_2,c)
=\frac{\alpha_1+\alpha_2}{2}(1+c)\,B_n(\alpha_1,\alpha_2,c),\quad 
(n\ge1),
\end{equation*}
where
\begin{equation}\label{Bn}
B_n(\alpha_1,\alpha_2,c):={}_2F_1\!\left(1-n,\,
1-\frac{\alpha_1+\alpha_2}{2};\,
2;\,
1+c
\right),\quad 
(n\ge1).
\end{equation}
Here, we recall the conventional Gaussian hypergeometric function by
\begin{equation*}
    {}_2F_1(a, b;c;z)=\sum_{n=0}^{\infty}\frac{(a)_n (b)_n}{(c)_n(1)_n}z^n,\quad (a,b,c\in\mathbb{C}, c\not\in\{0,-1,-2,\ldots\}, z\in\mathbb{D}).
\end{equation*}
The expression $(a)_n$, denoted as the Pochhammer symbol, signifies the product of consecutive terms starting from $a$ and ending at $a+n-1$, i.e., $(a)_n=a(a+1)\cdots(a+n-1)$ with the convention that $(a)_0$ is equal to $1$. For a deeper exploration of $G$ and its properties, refer to the comprehensive discussion in reference \cite[Section 2]{AP-2017}.

\begin{lemma}\label{lem-AP}\cite[Lemma 2.1]{AP-2017}
For $n\in\mathbb{N}$, let $B_n(\alpha_1,\alpha_2,c)$ be defined by \eqref{Bn}. Then for $|\alpha_1+\alpha_2|<2$, we have $|B_n(\alpha_1,\alpha_2,c)|\leq |B_n(\alpha_1,\alpha_2,1)|$. 
\end{lemma}

As a result of Lemma \ref{lem-AP}, we get the following:
\begin{corollary}
Let $A_n$ be defined as in \eqref{coef-An} for $n=1,2,3,\ldots$, and $|\alpha_1+\alpha_2|<2$ for $\alpha_1,\alpha_2\in(0,1]$. Then
    \begin{equation*}\label{est-An}
        |A_n|\leq \lambda\,|B_n(\alpha_1,\alpha_2,1)|,
    \end{equation*}
    where
\begin{equation}\label{lambda}
\lambda = (\alpha_1 + \alpha_2)\cos\left(\frac{\pi\theta}{2}\right)\quad\text{with} \quad 
\theta = \frac{\alpha_2 - \alpha_1}{\alpha_2 + \alpha_1}.
\end{equation}
\end{corollary}

The main purpose of this paper is to study the class $\mathcal{S}^*(\alpha_1,\alpha,_2)$, which is provided below:

\begin{definition}\label{def}{\rm 
Let $0<\alpha_1, \alpha_2\leq 1$. A function $f\in\mathcal{S}$ belongs to the class
$\mathcal{S}^*(\alpha_1,\alpha_2)$, if $f$ satisfies the following two-sided inequality
\begin{equation}\label{DS}
    -\frac{\pi\alpha_1}{2}< \arg\left\{\frac{zf'(z)}{f(z)}\right\}
    <\frac{\pi\alpha_2}{2}, \quad (z\in\mathbb{D}),
\end{equation}
or, equivalently, 
\begin{equation}\label{sub-con-def}
    \frac{zf'(z)}{f(z)}\prec G(z),
\end{equation}
where $G$ is defined as in \eqref{g}.}
\end{definition}
The introduction of the class $\mathcal{S}^*(\alpha_1, \alpha_2)$ is attributed to Takahashi and Nunokawa, specifically in the context of inequalities \eqref{DS}, see \cite{Takahashi}. We recall here the fact that in \cite{[BC1]} and
in \cite{[BC2]} a similar class was studied. It is clear that
$\mathcal{S}^*(\alpha_1,\alpha_2)\subset \mathcal{S}^*$ and that
$\mathcal{S}^*(\alpha_1,\alpha_2)$ is a subclass of the class of
strongly starlike functions of order $\beta=\max\{\alpha_1,
\alpha_2\}$, i.e.
$\mathcal{S}^*(\alpha_1,\alpha_2)\subset\mathcal{S}^*(\beta,\beta)\equiv
\mathcal{SS}^*(\beta)$.

 In recent decades, numerous researchers have explored various subclasses of analytic functions defined through the concept of subordination (see, e.g., \cite{KLS, KO2011, Men, RPK, RaiSok2015, Sok2011, SokSta}). Motivated by these developments, we study the class $\mathcal{S}^*(\alpha_1,\alpha_2)$ using the subordination relation \eqref{sub-con-def}. This approach facilitates the investigation of its geometric properties.

The following lemmas will be useful in this paper.
\begin{lemma}\label{Lem-MM}{\rm (See \cite{MM-book})}
Let $\lambda$ and $\gamma$ be complex numbers with $\lambda\neq 0$ and let $q$ be convex univalent in $\mathbb{D}$ with ${\rm Re}\{\lambda q(z) + \gamma \} \geq 0$. If $p$ is analytic in $\mathbb{D}$ with $p(0) = q(0)$, then
\begin{equation*}
    p(z)+\frac{zp'(z)}{\lambda p(z) + \gamma}\prec q(z)\Rightarrow p(z)\prec q(z).
\end{equation*}
\end{lemma}
\begin{lemma}\label{t0}
{\rm (See \cite{RUST})} Let $F,H\in \mathcal H$ be any convex univalent functions in
$\mathbb{D}$. If $f\prec F$ and $g\prec H$, then
\begin{equation}\label{1t0}
    f(z)*g(z)\prec F(z)*H(z),
\end{equation}
where ``*" denotes the well-known Hadamard product.
\end{lemma}
In the following, we recall some other useful lemmas.
\begin{lemma}[See \cite{Rog}]\label{lem1.3}
  Let $q(z)=\sum_{n=1}^{\infty}Q_nz^n$ be analytic and
univalent in $\mathbb{D}$, and suppose that $q$ maps $\mathbb{D}$ onto a
convex domain. If $p(z) = \sum_{n=1}^{\infty}P_nz^n$ is analytic in
$\mathbb{D}$ and satisfies the following subordination:
\begin{equation*}
    p(z)\prec q(z),
\end{equation*}
then
\begin{equation*}
    |P_n|\leq |Q_1|,\quad (n=1,2,3,\ldots).
\end{equation*}
\end{lemma}

\begin{lemma}\label{lem-Ali} \cite[Lemma 1]{ali2007}
If $w(z)=\sum_{k=1}^{\infty 
}w_kz^k$ is a Schwarz function, then
\begin{equation*}
  |w_2-t w_1^2|\leq \left\{
                      \begin{array}{ll}
                        -t, & \hbox{$t\leq -1$;} \\\\
                        1,  & \hbox{$-1\leq t\leq 1$;} \\\\
                        t, & \hbox{$t\geq1$.}
                      \end{array}
                    \right.
\end{equation*}
which simplifies to $|w_{2}+tw_{1}^{2}|\le \max \{1,|t|\}$ for $t \in \mathbb R$.
All inequalities are sharp.
\end{lemma}
The following lemma is due to Prokhorov and Szynal.
\begin{lemma}\label{Lemma ProSzyn}{\rm(See \cite{ProSzyn})}
If $w(z)=\sum_{n=1}^{\infty 
}w_nz^n$ is a Schwarz function, then for any complex numbers $\rho$ and $\tau$ the following sharp estimate holds:
\begin{equation*}\label{estimate ProSzyn}
    |w_3+\rho w_1w_2+\tau w_1^3|\leq H(\rho,\tau),
\end{equation*}
where
\begin{equation}\label{Hpq}
  H(\rho,\tau)= \left\{
                     \begin{array}{ll}
                     1, & \hbox{for $(\rho,\tau)\in \Omega_1\cup \Omega_2$;} \\\\
                     |\rho|, & \hbox{for $(\rho,\tau)\in \bigcup_{k=3}^7\Omega_k $;} \\\\
                     \frac{2}{3}(|\rho|+1)\left(  \frac{|\rho|+1}{3(|\rho|+1+\tau)}\right)^{\frac{1}{2}},& \hbox{for $(\rho,\tau)\in \Omega_8\cup \Omega_9$;} \\\\
                     \frac{\tau}{3}\left( \frac{\rho ^2-4}{\rho^2-4\tau} \right)^{\frac{1}{2}},  & \hbox{for $(\rho,\tau)\in \Omega_{10}\cup \Omega_{11} \setminus \{\pm2,1\}$;} \\\\
                     \frac{2}{3}(|\rho|-1)\left(  \frac{|\rho|-1}{3(|\rho|-1-\tau)}\right)^{\frac{1}{2}},& \hbox{for $(\rho,\tau)\in \Omega_{12}.$} \\\\
                      \end{array}
                    \right.
\end{equation}
The extremal functions, up to rotations, are of the form
\begin{equation*}
  w(z)=z^3, \quad  w(z)=z, \quad  w(z)= w_0(z)=\frac{[(1-\lambda)\epsilon_2+\lambda \epsilon_1]z-\epsilon_1\epsilon_2z}{1-[(1-\lambda)\epsilon_1+\lambda\epsilon_2]z},
\end{equation*}

\begin{equation*}
  w(z)= w_1(z)=\frac{z(t_1-z)}{1-t_1z}, \quad  w(z)= w_2(z)=\frac{z(t_2+z)}{1+t_2z},
\end{equation*}

\begin{equation*}
  |\epsilon_1|=|\epsilon_2|=1, \quad \epsilon_1=t_0-e^{\frac{-i\theta_0}{2}}(a\mp b), \quad \epsilon_2=-e^{\frac{-i\theta_0}{2}}(ia\pm b),
\end{equation*}

\begin{equation*}
  a=t_0\cos\frac{\theta_0}{2},\quad b=\sqrt{1-t_0^2\sin^2\frac{\theta_0}{2}}, \quad \lambda=\frac{b\pm a}{2b},
\end{equation*}

\begin{equation*}
  t_0=\left(\frac{2\tau(\rho^2+2)-3\rho^2}{3(\tau-1)(\rho^2-4\tau)}   \right)^{\frac{1}{2}}, \quad t_1=\left( \frac{|\rho|+1}{3(|\rho|+1+\tau)}  \right)^{\frac{1}{2}},
\end{equation*}

\begin{equation*}
  t_2=\left( \frac{|\rho|-1}{3(|\rho|-1-\tau)}  \right)^{\frac{1}{2}}, \quad \cos \frac{\theta_0}{2}=\frac{\rho}{2}\left[\frac{\tau(\rho^2+8)-2(\rho^2+2)}{2\tau(\rho^2+2)-3\rho^2}  \right].
\end{equation*}
The sets $\Omega_i$, $i=1,2,\dots,12$ are defined as follows:
\begin{equation*}
  \Omega_1=\left\{(\rho,\tau):|\rho|\leq \frac{1}{2}, |\tau|\leq 1 \right\},
\end{equation*}

\begin{equation*}
  \Omega_2=\left\{(\rho,\tau):\frac{1}{2}\leq |\rho|\leq 2, \frac{4}{27}(|\rho|+1)^3-(|\rho|+1)\leq \tau \leq 1 \right\},
\end{equation*}

\begin{equation*}
  \Omega_3=\left\{(\rho,\tau):|\rho|\leq \frac{1}{2}, \tau\leq -1 \right\},
\end{equation*}

\begin{equation*}
  \Omega_4=\left\{(\rho,\tau):|\rho|\geq \frac{1}{2}, \tau \leq -\frac{2}{3}(|\rho|+1) \right\},
\end{equation*}

\begin{equation*}
  \Omega_5=\left\{(\rho,\tau):|\rho|\leq 2, \tau \geq 1 \right\},
\end{equation*}

\begin{equation*}
  \Omega_6=\left\{(\rho,\tau):2 \leq |\rho|\leq 4, \tau \geq \frac{1}{12}(\rho^2+8) \right\},
\end{equation*}

\begin{equation*}
  \Omega_7=\left\{(\rho,\tau):|\rho|\geq 4, \tau \geq \frac{2}{3}(|\rho|-1) \right\},
\end{equation*}

\begin{equation*}
  \Omega_8=\left\{(\rho,\tau):\frac{1}{2}\leq |\rho|\leq 2, -\frac{2}{3}(|\rho|+1)\leq \tau \leq \frac{4}{27}(|\rho|+1)^3-(|\rho|+1) \right\},
\end{equation*}

\begin{equation*}
  \Omega_9=\left\{(\rho,\tau):|\rho|\geq 2, -\frac{2}{3}(|\rho|+1)\leq \tau \leq \frac{2|\rho|(|\rho|+1)}{\rho^2+2|\rho|+4} \right\},
\end{equation*}

\begin{equation*}
  \Omega_{10}=\left\{(\rho,\tau):2 \leq |\rho|\leq 4, \frac{2|\rho|(|\rho|+1)}{\rho^2+2|\rho|+4}\leq\tau \leq \frac{1}{12}(\rho^2+8)\right\},
\end{equation*}

\begin{equation*}
  \Omega_{11}=\left\{(\rho,\tau):|\rho|\geq 4, \frac{2|\rho|(|\rho|+1)}{\rho^2+2|\rho|+4}\leq \tau \leq \frac{2|\rho|(|\rho|-1)}{\rho^2-2|\rho|+4} \right\},
\end{equation*}

\begin{equation*}
  \Omega_{12}=\left\{(\rho,\tau):|\rho|\geq 4, \frac{2|\rho|(|\rho|-1)}{\rho^2-2|\rho|+4}\leq \tau \leq \frac{2}{3}(|\rho|-1) \right\}.
\end{equation*}
\end{lemma}

By Lemma \ref{Lem-MM}, we derive a condition that is sufficient for functions to be members of the class $\mathcal{S}^*(\alpha_1, \alpha_2)$.
\begin{lemma}
  If the function $f\in \mathcal{S}$ satisfies the following subordination,
  \begin{equation*}\label{1+zf'' f'}
    1+\frac{zf''(z)}{f'(z)}\prec G(z),
  \end{equation*}
  then $f$ belongs to the class $\mathcal{S}^*(\alpha_1,\alpha_2)$,
  where $G$ is given by \eqref{g}.
\end{lemma}
\begin{proof}
  Let $f\in\mathcal{S}$ and 
  \begin{equation}\label{1+p(z)}
    p(z)=\frac{zf'(z)}{f(z)},\quad(z\in\mathbb{D}).
  \end{equation}
  Then $p$ is analytic and $p(0)=1$. Upon taking the logarithmic derivative of equation \eqref{1+p(z)} with respect to $z$, we obtain:
  \begin{equation*}
    p(z)+\frac{zp'(z)}{p(z)}=1+\frac{zf''(z)}{f'(z)},\quad(z\in\mathbb{D}).
  \end{equation*}
  Since $G$ is convex univalent in $\mathbb{D}$ and ${\rm Re} \{G(z)\}>0$, the conclusion can be derived directly from Lemma \ref{Lem-MM}.
\end{proof}


\section{Main Results}\label{sec-Main Results}
Determining the sharp coefficient bounds for strongly starlike functions appears to pose a challenging classical problem. To date, the estimation of the coefficients of these functions remains an open problem. There are several known results for specific subclasses of strongly starlike functions, but there is no general solution for the entire class. One of the main challenges in this problem is that the definition of strongly starlike functions is quite broad, and there is a wide variety of functions that fall into this category. This makes it difficult to find a single method that can be used to estimate the coefficients of all strongly starlike functions. Another challenge is that the coefficients of strongly starlike functions can be quite complex, and there is no easy way to calculate them exactly. This means that any estimates that are obtained are likely to be approximate. Despite these challenges, there has been some progress on this problem in recent years. Notably, only the initial four coefficients have known accurate estimates, as documented by Brannan et al. \cite{BCK}, Ali and Singh \cite{AliS}, and Lecko and Sim \cite{LS2017}, see Theorem \ref{thm-u b} below. 
\begin{theorem}\label{thm-u b}
Let the function $f\in\mathcal{A}$ belong to the class $\mathcal{SS}^*(\beta)$, where $\beta\in (0,1]$. Then the following sharp inequalities hold:
\begin{equation*}
    |a_2|\leq 2\beta,\qquad |a_3|\leq \left\{
  \begin{array}{ll}
    \beta, & \hbox{$0<\beta \leq 1/3$;} \\\\
    3\beta^2, & \hbox{$1/3\leq \beta \leq 1$,}
  \end{array}
\right.
\end{equation*}
and
\begin{equation*}
    |a_4|\leq \left\{
  \begin{array}{ll}
    \frac{2\beta}{3}, & \hbox{$0<\beta \leq \sqrt{2/17}$;} \\\\
    \frac{2\beta}{9}(1+17\beta^2), & \hbox{$\sqrt{2/17}\leq \beta \leq 1$.}
  \end{array}
\right.
\end{equation*}
\end{theorem}
To see the sharp upper bound for $|a_5|$, refer to \cite{AliS}. In the case where $\beta\in(0,1]$ is a real number, refer to \cite[Theorem 2]{Kwon}. It is important to note that an error was identified in the final segment of the proof concerning the estimate of $|a_4|$ in \cite{AliS}, and Lecko and Sim subsequently corrected this error; see \cite{LS2017}. These efforts are hoped to eventually lead to a general solution for a sharp estimation of the coefficients of strongly starlike functions.

In this paper, we find a general upper bound for $|a_n|$ of the function $f\in \mathcal{S}^*(\alpha_1,\alpha_2)$ and $f\in \mathcal{SS}^*(\beta)$, where $n=2,3,\ldots$.
Our findings generally do not produce sharp results, except when $n=2$. Of course, in Theorem \ref{Th. a2a3a4} below, we obtain sharp estimates for the first three coefficients of $f\in \mathcal{S}^*(\alpha_1,\alpha_2)$.

Let us start by estimating the coefficients of the function $f\in \mathcal{S}^*(\alpha_1,\alpha_2)$.

\begin{theorem}\label{t22}
Let $f$ be of the form \eqref{f} belonging to the class $\mathcal{S}^*(\alpha_1,\alpha_2)$, where $0 < \alpha_1, \alpha_2 \leq 1$.
Then $|a_2|\leq \lambda$ and for $n \geq 3$,
\begin{equation}\label{an}
|a_n| \leq \frac{\lambda}{n-1} \prod_{k=2}^{n-1}\left(1 + \frac{\lambda}{k-1}\right),
\end{equation}
where $\lambda$ is given by \eqref{lambda}. 
The bound is sharp for $n = 2$. 
\end{theorem}
\begin{proof}
Consider the function $q$ as follows:
\begin{equation}\label{1t3}
   zf'(z)=q(z)f(z),\quad (z\in\mathbb{D}).
\end{equation}
Then, by Definition \ref{def}, we have
\begin{equation}\label{2t3}
   q(z)\prec G(z),\quad (z\in\mathbb{D}),
\end{equation}
where $G$ is defined by \eqref{g}. If we let
\begin{equation*}\label{3t3}
   q(z)=1+\sum_{n=1}^{\infty} Q_nz^n,
\end{equation*}
then, by Lemma \ref{lem1.3}, we see that the subordination relation
\eqref{2t3} implies that
\begin{equation}\label{4t3}
|Q_n|\leq |A_1|=(\alpha_1+\alpha_2)\cos\left(\frac{\pi\theta}{2}\right)=\lambda,\quad (n=1,2,3,\ldots).
\end{equation}
If we equate the coefficients of $z^n$ on both sides of
\eqref{1t3}, we obtain
\begin{equation*}
   na_n=Q_{n-1} a_1+Q_{n-2}a_2+\cdots+Q_1a_{n-1}+Q_0a_n,\quad
   (n=2,3,\ldots),
\end{equation*}
where $Q_0=a_1=1$. With a simple calculation and also by \eqref{4t3}, we get
\begin{align*}
  |a_n|&=\frac{1}{n-1}\times \left|Q_{n-1}a_1+Q_{n-2}a_2+\cdots+Q_1a_{n-1}\right|\\
  &\leq\frac{\lambda}{n-1}(|a_{1}|+|a_2|+\cdots+|a_{n-1}|)\\
  &=\frac{\lambda}{n-1}\sum_{k=1}^{n-1}|a_k|.
\end{align*}
It is clear that $|a_2|\leq \lambda$. To prove the remaining part of the theorem, we need to show that
\begin{equation}\label{induction}
  \frac{\lambda}{n-1}\sum_{k=1}^{n-1}|a_k|\leq \frac{\lambda}{n-1}
  \prod_{k=2}^{n-1}\left(1+\frac{\lambda}{k-1}\right),
\end{equation}
for $n=3,4,\ldots$. Using induction and simple calculation, we could prove the inequality \eqref{induction}. Hence, the desired estimate for
$|a_n|$, $(n = 3, 4, 5,\ldots)$, follows as asserted in \eqref{an}.

It is easy to see that the extremal function
\begin{equation}\label{f_extremal}
    f(z)
    =
    z \exp\!\left(
    \int_0^z \frac{G(w(\zeta))-1}{\zeta}\, d\zeta
    \right),
\end{equation}
gives $|a_2| = \lambda$ exactly, where $w$ is a Schwarz function.
For $n \geq 3$, the bounds are generally not sharp because of the compounding effect of the subordination chain.
This completes the proof of the theorem.
\end{proof}
Selecting $\alpha_1=\alpha_2=\beta$, in the above Theorem \ref{t22}, we may obtain bounds on the coefficients of a strongly starlike function of order $\beta$. Moreover, they are sharp when $\beta=1$ or $\beta\rightarrow 0+$, for $n=3,4,\ldots$, as shown in Figure \ref{fig:upper bounds for a3} for $n=3$.
\begin{figure}
    \centering
    \includegraphics[width=10cm]{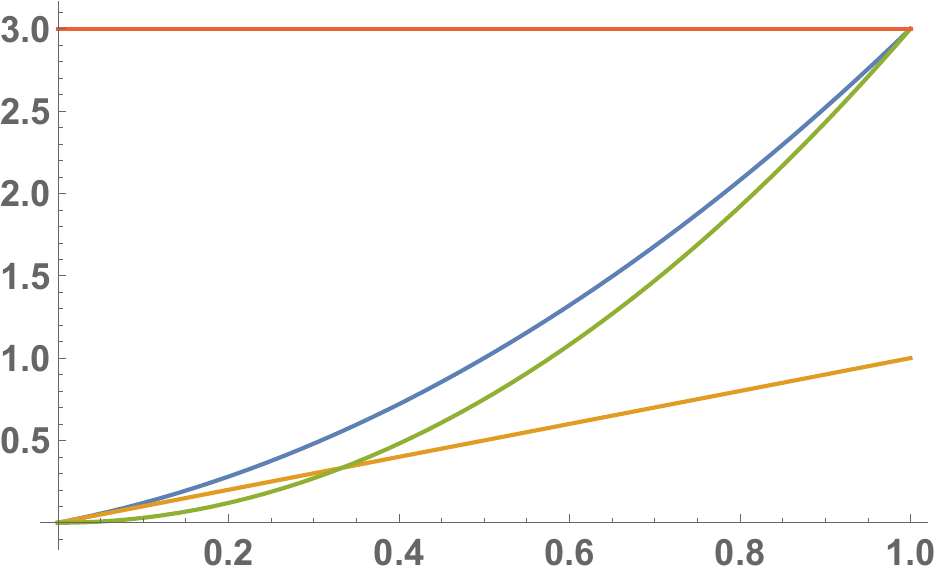}
    \caption{Upper bound for $|a_3|$ by Theorem \ref{thm-upper b for SS} (blue), and by Theorem \ref{thm-u b} (green and orange), where $\beta\in(0,1]$. The red line shows an upper bound for $|a_3|$ of a starlike univalent function.}
    \label{fig:upper bounds for a3}
\end{figure}

\begin{theorem}\label{thm-upper b for SS}
  If the function $f$ of the form \eqref{f} is a strongly starlike function of order $\beta$, then
  \begin{equation*}
  |a_n|\leq\left\{
  \begin{array}{ll}
    2\beta, & \qquad\hbox{$n=2$;} \\ \\
    \frac{2\beta}{n-1}
   \prod_{k=2}^{n-1}\left(1+\frac{2\beta}{k-1}\right), & \qquad\hbox{$n=3,4,\ldots$,}
\end{array}%
\right.
\end{equation*}
where $0<\beta\leq 1$. The bound is sharp for $n = 2$. 
\end{theorem}


Next, applying the lemma \ref{lem-Ali} and Lemma \ref{Lemma ProSzyn}, we proceed to determine the first three coefficients of the function $f$ represented in the form \eqref{f}, which falls within the class $\mathcal{S}^*(\alpha_1, \alpha_2)$.
\begin{theorem}\label{Th. a2a3a4}
Consider a function $f\in \mathcal{A}$, which takes the form \eqref{f} and belongs to the class $\mathcal{S}^*(\alpha_1, \alpha_2)$, where $\alpha_1,\alpha_2\in(0,1]$. Let $\lambda$ be defined by \eqref{lambda}, and $\gamma=(3\alpha-1)/2$, where $\alpha$ is given by \eqref{alpha}.
Then
\begin{equation*}
|a_2|\leq \lambda,\qquad |a_3|
\le \left\{
  \begin{array}{ll}
    \frac{\lambda}{2}, & \qquad\hbox{$0<\alpha\le 1/3$;} \\ \\
    \frac{\lambda}{2}\sqrt{1+4\gamma(1+\gamma)\cos^2\!\left(\frac{\pi\theta}{2}\right)}, & \qquad\hbox{$1/3\le \alpha\le 1$,}
\end{array}%
\right.
\end{equation*}
and
\begin{equation*}
    |a_4|\leq
    \frac{\lambda}{3}\, H(\rho,\tau),
\end{equation*}
where $H(\rho,\tau)$ is given by~\eqref{Hpq} with
\begin{equation}\label{eq:rhotau}
\rho = \frac{3}{2}A_1 + \frac{2A_2}{A_1}, \quad \tau = \frac{1}{2}A_1^2 + \frac{3}{2}A_2 + \frac{A_3}{A_1}.
\end{equation}
Each inequality is sharp.
\end{theorem}
\begin{proof}
Let $\alpha_1$ and $\alpha_2$ belong to the interval $(0,1]$. If the function $f\in\mathcal{A}$ belongs to the class $\mathcal{S}^*(\alpha_1, \alpha_2)$, then by using the Definition \ref{def}, there exists a Schwarz function $w(z)=w_1 z+w_2 z^2+w_3 z^3+\cdots$ such that
\begin{equation}\label{eq-zf}
    \frac{zf'(z)}{f(z)}=G(w(z)),\quad (z\in\mathbb{D}).
\end{equation}
Upon substituting the Taylor series expansions of both sides of \eqref{eq-zf}, we derive the following:
\begin{align*}
    &1+a_2z+(2a_3-a_2^2)z^2+(3a_4-3a_2a_3+a_2^3)z^3+\cdots\\ &\quad=1+A_1 w_1 z+(A_1w_2+A_2w_1^2)z^2+(A_1w_3+2A_2w_1w_2+A_3w_1^3)z^3+\cdots,
  \end{align*}
where $A_n$, $n=1,2,3$, are defined in \eqref{A123}. Equating the coefficients of the corresponding terms in the last relation leads to
\begin{equation}\label{a2a3}
a_2=A_1w_1,\quad 2a_3=A_1w_2+(A_2+A_1^2)w_1^2,
\end{equation}
  and 
  \begin{equation}\label{a4}
     3 a_4= A_1 \left[w_3+\left(\frac{3}{2}A_1+\frac{2A_2}{A_1}\right)w_1w_2+\left(\frac{1}{2}A_1^2+\frac{3}{2}A_2+\frac{A_3}{A_1}\right)w_1^3\right].
  \end{equation}
It follows from the first equality in \eqref{a2a3} that $|a_2| = |A_1||w_1| \leq |A_1| = \lambda$, which yields the sharp upper bound for $|a_2|$.

The second coefficient relation in \eqref{a2a3} can be rewritten as
\begin{equation*}
2a_3
=
A_1\left(
w_2
+
\left(
\frac{A_2}{A_1}+A_1
\right) w_1^2
\right).
\end{equation*}
An application of Lemma~\ref{lem-Ali} yields
\begin{equation}\label{2|a3|}
2|a_3|
\le
|A_1|\max \left\{1,
\left|
\frac{A_2}{A_1}+A_1
\right|\right\}.
\end{equation}
A direct computation using \eqref{A123} shows that
\begin{equation*}
\frac{A_2}{A_1}+A_1
=
1+\frac{3\alpha-1}{2}(1+c)=
1+\gamma(1+c).
\end{equation*}
Taking moduli gives
\begin{equation*}
\left|
\frac{A_2}{A_1}+A_1
\right|
=
\sqrt{1+4\gamma(1+\gamma)\cos^2\!\left(\frac{\pi\theta}{2}\right)}.
\end{equation*}
It is easy to see that $|A_2/A_1+A_1|\geq 1$ if and only if $\gamma\ge 0$, which implies that $\alpha\ge 1/3$. Therefore, we have
\begin{equation*}
\max\left\{1, \left|
\frac{A_2}{A_1}+A_1
\right|\right\}=\sqrt{1+4\gamma(1+\gamma)\cos^2\!\left(\frac{\pi\theta}{2}\right)}.
\end{equation*}
On the other hand, $|A_2/A_1+A_1|\le  1$ if and only if $\gamma\le  0$, which implies that $\alpha\le  1/3$. Thus, 
\begin{equation*}
\max\left\{1, \left|
\frac{A_2}{A_1}+A_1
\right|\right\}=1,
\end{equation*}
if $0<\alpha\le 1/3$.
Consequently, by \eqref{2|a3|} we obtain the sharp bound for $|a_3|$.

From the coefficient relation \eqref{a4}
\begin{equation*}\label{a4_formula}
    3a_4 = A_1\left[ w_3 + \rho\, w_1w_2 + \tau\, w_1^3 \right],
\end{equation*}
where
\begin{equation*}
\rho = \frac{3}{2}A_1 + \frac{2A_2}{A_1}, \quad \tau = \frac{1}{2}A_1^2 + \frac{3}{2}A_2 + \frac{A_3}{A_1}.
\end{equation*}
An application of Lemma~\ref{Lemma ProSzyn} yields
\begin{equation*}\label{a4_bound_general}
    |a_4| \leq \frac{|A_1|}{3}\, H(\rho,\tau)=\frac{\lambda}{3}\, H(\rho,\tau),
\end{equation*}
where $H(\rho,\tau)$ is given by~\eqref{Hpq}.

The coefficient inequalities in Theorem~\ref{Th. a2a3a4} are attained by different extremal functions. 
More precisely, each bound is realized by choosing an appropriate extremal Schwarz function 
$w$ in the subordination relation
\begin{equation}\label{eq:extremal_general}
    \frac{z f'(z)}{f(z)} = G(w(z)),
\end{equation}
where $G$ is given by~\eqref{g}. 
For any such Schwarz function $w$, the corresponding extremal function is uniquely determined by the extremal function \eqref{f_extremal}.
The sharp bound for $|a_2|$ is attained when $w(z)=e^{i\phi}z$ for some $\phi\in\mathbb{R}$. Therefore, the corresponding extremal function is given by:
\begin{equation*}
f_2(z)
=
z
+ A_1 e^{i\phi} z^2
+ \frac{A_2+A_1^2}{2} e^{2i\phi} z^3
+ \frac{A_3+2A_1A_2+A_1^3}{3} e^{3i\phi} z^4
+ \cdots.
\end{equation*}

The sharpness of the bound for $|a_3|$ depends on the range of the parameter $\alpha$.
If $0<\alpha\le 1/3$, then
\begin{equation*}
\max\left\{1,\left|\frac{A_2}{A_1}+A_1\right|\right\}=1,
\end{equation*}
and equality in Lemma~\ref{lem-Ali} is attained for Schwarz functions of the form
$w(z)=e^{i\psi}z^2$, $\psi\in\mathbb{R}$.  
The corresponding extremal function satisfies
\begin{equation*}
f_{31}(z)
=
z
+
\frac{A_1}{2} e^{i\psi} z^3
+
O(z^5).
\end{equation*}

If $1/3\le \alpha\le 1$, then
\begin{equation*}
\max\left\{1,\left|\frac{A_2}{A_1}+A_1\right|\right\}
=
\left|\frac{A_2}{A_1}+A_1\right|,
\end{equation*}
and equality is attained for Schwarz functions of the form
$w(z)=e^{i\phi}z$, $\phi\in\mathbb{R}$.  
In this case, the extremal function is given by
\begin{equation*}
f_{32}(z)
=
z
+
A_1 e^{i\phi} z^2
+
\frac{A_2+A_1^2}{2} e^{2i\phi} z^3
+
\cdots.
\end{equation*}

Finally, the sharp bound for $|a_4|$ is attained by Schwarz functions realizing equality in the Prokhorov--Szynal lemma, which are finite Blaschke products of degree at most two. In particular, equality is attained for Schwarz functions of the form
\begin{equation*}
w(z)=e^{i\phi}\frac{z-a}{1-\overline{a}z}, \quad |a|\le 1,
\end{equation*}
for which equality holds in the Prokhorov--Szynal lemma.

Consequently, there is no single extremal function that simultaneously attains equality in all three coefficient bounds.
This completes the proof.
\end{proof}

\begin{remark}
We illustrate the bound for $|a_4|$ in Theorem~\ref{Th. a2a3a4} for several special choices of
the parameters $\alpha_1$ and $\alpha_2$.

\medskip
\noindent
\textbf{Case 1.} Let $\alpha_1=\alpha_2=1$.  
Then $\theta=0$, $c=1$, and $\alpha=(\alpha_1+\alpha_2)/2=1$.  
Using \eqref{A123}, we obtain
\begin{equation*}
A_1=2,\qquad A_2=2,\qquad A_3=2.
\end{equation*}
Consequently,
\begin{equation*}
\rho=\frac{3}{2}A_1+\frac{2A_2}{A_1}=5,
\qquad
\tau=\frac{1}{2}A_1^2+\frac{3}{2}A_2+\frac{A_3}{A_1}=6.
\end{equation*}
Since $(\rho,\tau)\in\Omega_7$, Lemma~\ref{Lemma ProSzyn} yields
$H(\rho,\tau)=|\rho|=5$. Therefore,
\begin{equation*}
|a_4|
\le
\frac{\lambda}{3}H(\rho,\tau)
=
\frac{2}{3}\cdot 5
=
\frac{10}{3}.
\end{equation*}

\medskip
\noindent
\textbf{Case 2.} Let $\alpha_1=\alpha_2=\tfrac12$.  
Then $\theta=0$, $c=1$, and $\alpha=\tfrac12$.  
From \eqref{A123}, we find
\begin{equation*}
A_1=1,\qquad A_2=\tfrac12,\qquad A_3=\tfrac12.
\end{equation*}
Hence,
\begin{equation*}
\rho=\frac{3}{2}A_1+\frac{2A_2}{A_1}=\frac{5}{2},
\qquad
\tau=\frac{1}{2}A_1^2+\frac{3}{2}A_2+\frac{A_3}{A_1}=\frac{7}{4}.
\end{equation*}
Since $(\rho,\tau)\in\Omega_6$, Lemma~\ref{Lemma ProSzyn} gives
$H(\rho,\tau)=|\rho|=\tfrac52$. Thus,
\begin{equation*}
|a_4|
\le
\frac{\lambda}{3}H(\rho,\tau)
=
\frac{1}{3}\cdot\frac{5}{2}
=
\frac{5}{6}.
\end{equation*}

\medskip
\noindent
\textbf{Case 3.}
Let $\alpha_1=1$ and $\alpha_2=\tfrac12$.  
Then $\alpha=\tfrac34$ and $\theta=-\tfrac13$, so that
$c=e^{-i\pi/3}$ and $|1+c|=\sqrt{3}$.  
Using \eqref{A123}, we obtain
\begin{equation*}
A_1=\tfrac34(1+c),\quad
A_2=\tfrac34(1+c)-\tfrac{3}{32}(1+c)^2,\quad
A_3=\tfrac34(1+c)-\tfrac{3}{16}(1+c)^2+\tfrac{5}{128}(1+c)^3.
\end{equation*}
A direct computation yields
\begin{equation*}
\rho=\frac{3}{2}A_1+\frac{2A_2}{A_1}\approx 3.79,
\qquad
\tau=\frac{1}{2}A_1^2+\frac{3}{2}A_2+\frac{A_3}{A_1}\approx 3.02.
\end{equation*}
Since $(\rho,\tau)\in\Omega_6$, Lemma~\ref{Lemma ProSzyn} again implies
$H(\rho,\tau)=|\rho|$. Hence,
\begin{equation*}
|a_4|
\le
\frac{\lambda}{3}H(\rho,\tau)
\approx 1.64.
\end{equation*}

\medskip
\noindent
\textbf{Remark on boundary behavior.}
If $\alpha_1+\alpha_2\to 0$, then $A_1\to 0$, and the parameters $\rho$ and $\tau$
are no longer well defined. Therefore, such limiting cases fall outside the
scope of Theorem~\ref{Th. a2a3a4} and are excluded from consideration.

\medskip
In all admissible cases, the bound for $|a_4|$ is sharp in the sense of
Lemma~\ref{Lemma ProSzyn}, and equality is attained by extremal Schwarz
functions given by finite Blaschke products of degree at most two.
\end{remark}

\subsection{Logarithmic Coefficients.}
Logarithmic coefficients, denoted by $\gamma_n$, are extracted from the function $f(z)$ using the following series representation:
\begin{equation}\label{log coef}
  \log\left\{\frac{f(z)}{z}\right\}=\sum_{n=1}^{\infty}2\gamma_n z^n,\quad (z\in \mathbb{D}).
\end{equation}
These coefficients, $\gamma_n$, hold a significant position in various estimations within the realm of univalent function theory. For example, if $f\in \mathcal{S}$, then we have
\begin{equation*}
  \gamma_1=\frac{a_2}{2},\quad{\rm and} \quad \gamma_2=\frac{1}{2}\left(a_3-\frac{a_2^2}{2}\right)
\end{equation*}
and the sharp estimates
\begin{equation*}
  |\gamma_1|\leq1, \quad{\rm and}\quad |\gamma_2|\leq \frac{1}{2}(1+2e^{-2})\approx 0.635\ldots,
\end{equation*}
hold. For $n \geq 3$, estimating $\gamma_n$ is much more difficult, and no significant upper bounds for $|\gamma_n|$ are known when $f \in \mathcal{S}$; the problem remains open for $n \geq 3$.
The sharp upper bounds for the modulus of logarithmic coefficients have been established for only a limited number of function subclasses within $\mathcal{S}$; the class $\mathcal{S}^*$ offers a relatively straightforward proof that $|\gamma_n|\leq 1/n$ for $n\geq1$, with equality observed in the Koebe function. For other subclasses within $\mathcal{S}$, further details can be found in references \cite{kargarJAnal,PWS1,PWS2,thomas}.

In what follows, we focus on estimating the logarithmic coefficients of functions belonging to the class $\mathcal{S}^*(\alpha_1,\alpha_2)$. First, we need the following theorem:
\begin{theorem}\label{t1}
If the function $f\in \mathcal{A}$ belongs to the class $\mathcal S^*(\alpha_1,\alpha_2)$, then
\begin{equation}\label{0t1}
        \log\left\{\frac{f(z)}{z}\right\}\prec\int_0^z \frac{G(t)-1}{t}{\rm d}t,
\end{equation}
where the function $G$ is defined by \eqref{g}. Moreover,
\begin{equation}\label{wide g}
         \widetilde{G}(z):=\int_0^z \frac{G(t)-1}{t}{\rm d}t,\quad(z\in \mathbb{D}),
\end{equation}
is a convex univalent function.
\end{theorem}
\begin{proof}
The subordination relation \eqref{sub-con-def} implies that
\begin{equation}\label{2t1}
   z\left(\log\left\{\frac{f(z)}{z}\right\}\right)'\prec G(z)-1,
\end{equation}
where $G(z)-1$ is convex univalent. For $x\geq0$
the function
\begin{equation*}
    \tilde{h}(x;z)=\sum\limits_{k=1}^\infty\frac{(1+x)z^k}{k+x},
\end{equation*}
is convex univalent in $\mathbb{D}$ (see \cite{RU2}). 
By \eqref{1t0}, we have:
\begin{equation}\label{3t1}
    \left[g(z)\prec F(z)\right]\Rightarrow\left[g(z)*\tilde{h}(0;z)\prec F(z)*\tilde{h}(0;z)\right],
\end{equation}
whenever $F$ is a convex univalent function. Since
\begin{equation}\label{4t1}
    g(z)*\tilde{h}(0;z)=\int_0^z \frac{g(t)}{t}{\rm d}t,
\end{equation}
therefore, by \eqref{2t1}-\eqref{3t1} and \eqref{4t1}, we get
\begin{equation*}
     \int_0^z \log\left\{f(t)\right\}'{\rm d}t\prec\int_0^z \frac{G(t)-1}{t}{\rm d}t,
\end{equation*}
which gives \eqref{0t1}. Moreover,
\begin{equation*}
    \widetilde{G}(z)=\int_0^z \frac{G(t)-1}{t}{\rm
    d}t=\left\{G(z)-1\right\}*\tilde{h}(0;z).
\end{equation*}
Since the class of convex univalent functions is known to be preserved under convolution, as demonstrated in \cite{RUSS}, we can confidently conclude that the function $\widetilde{G}$ also falls within the class of convex univalent functions. This completes the proof.
\end{proof}
Applying the last theorem, we get the following.
\begin{theorem}\label{t2}
Let $\widetilde{G}$ be of the form \eqref{wide g} and $0<\alpha_1, \alpha_2\leq 1$.
If $f(z)\in\mathcal S^*(\alpha_1,\alpha_2)$, then
\begin{equation*}
r\exp\widetilde{G}(-r)\leq \left|f(z)\right|
\leq r\exp\widetilde{G}(r),
\end{equation*}
for each $r=|z|<1$. Both inequalities are sharp.
\end{theorem}
\begin{theorem}\label{th.logcoef}
Let $f\in \mathcal{S}^*(\alpha_1,\alpha_2)$ and the coefficients of $\log(f(z)/z)$ be given by \eqref{log coef} and $0<\alpha_1, \alpha_2\leq 1$. Then
\begin{equation}\label{gamma_n}
  |\gamma_n|\leq \frac{\lambda}{2n},\quad(n\ge 1),
\end{equation}
where $\lambda$ is given by \eqref{lambda}.
The result is sharp.
\end{theorem}
\begin{proof}
  Consider $f\in \mathcal{S}^*(\alpha_1,\alpha_2)$. By the Taylor series of $G$, \eqref{log coef}, and \eqref{2t1}, we obtain
\begin{equation*}
        \sum_{n=1}^{\infty}2n\gamma_n z^n\prec \sum_{n=1}^{\infty}A_n z^n,
\end{equation*}
where $A_n$ is defined in \cite[(2.3)]{AP-2017}.
Utilizing Lemma \ref{lem1.3} yields:
\begin{equation*}
  2n|\gamma_n|\leq | A_1|=\lambda.
\end{equation*}
Thus, the desired inequality \eqref{gamma_n} follows.
Equality holds for the logarithmic coefficients of the function
\begin{equation*}
z \mapsto z \exp \widetilde{G}(z),
\end{equation*}
where $\widetilde{G}$ is defined in \eqref{wide g}. This completes the proof.
\end{proof}
If we let $\alpha_1=\alpha_2=\beta$ in the above Theorem \ref{th.logcoef}, we get the following result, which was previously obtained by Elhosh, see \cite[Theorem 2]{El}.
\begin{corollary}
Let $f$ be a strongly starlike function of order $\beta$, where $0<\beta\leq 1$. Then the logarithmic coefficients of $f$ satisfy the following sharp inequality:
\begin{equation*}
  |\gamma_n|\leq \frac{1}{n}\beta,\quad(n\ge 1).
\end{equation*}
In particular, taking $\beta=1$ gives us an estimate of the logarithmic coefficients of starlike functions.
\end{corollary}

\subsection{Upper and Lower Bounds for \texorpdfstring{${\rm Re}\{zf'(z)/f(z)\}$, where \texorpdfstring$f\in S^*(\alpha_1,\alpha_2)$}{K}} We finish this paper by estimating ${\rm Re}\{zf'(z)/f(z)\}$, where $f\in S^*(\alpha_1,\alpha_2)$.

\begin{theorem}\label{t21}
Suppose that $f$ belongs to the class $\mathcal{A}$. If $f\in \mathcal{S}^*(\alpha_1,\alpha_2)$, then
\begin{equation*}\label{lower_bound}
    \mathrm{Re}\left\{\frac{zf'(z)}{f(z)}\right\}
    \geq
    \left(\frac{1-\left(1+2\cos(\pi \theta/2)\right)r}{1-r}\right)^{(\alpha_1+\alpha_2)/2}, \quad 0\leq|z|=r\leq\frac{1}{1+2\cos(\pi \theta/2)},
\end{equation*}
and
\begin{equation*}\label{upper_bound}
    \mathrm{Re}\left\{\frac{zf'(z)}{f(z)}\right\}
    \leq
    \left(\frac{1+\left(2\cos(\pi \theta/2)-1\right)r}{1-r}\right)^{(\alpha_1+\alpha_2)/2}, \quad
    0\leq|z|=r<1,
\end{equation*}
where $0<\alpha_1,\alpha_2\leq 1$ and $\theta=(\alpha_2-\alpha_1)/(\alpha_2+\alpha_1)$. 
Moreover, $\mathrm{Re}\{zf'(z)/f(z)\} > 0$ for all $z\in\mathbb{D}\setminus\{0\}$.
\end{theorem}
\begin{proof}
Let the function $f\in\mathcal{A}$ be in the class $\mathcal
S^*(\alpha_1,\alpha_2)$. Then, by the definition of subordination, there exists a Schwarz function $w(z)$,
satisfying the following conditions:
\begin{equation*}
    w(0)=0, \quad {\rm and}\quad |w(z)|<1, \quad (z\in\mathbb{D})
\end{equation*}
and such that
\begin{equation*}
    \frac{zf'(z)}{f(z)}=\left(\frac{1+cw(z)}{1-w(z)}\right)^{(\alpha_1+\alpha_2)/2},\quad
(z\in\mathbb{D}).
\end{equation*}
Define,
\begin{equation}\label{F}
    F(z):=\frac{1+cw(z)}{1-w(z)},\quad (z\in\mathbb{D}).
\end{equation}
From \eqref{F} we have
\begin{equation*}
  \left|F(z)-1\right|=\left|\frac{(1+c)w(z)}{1-w(z)}\right|\leq \frac{2|w(z)|\cos(\pi \theta/2)}{1-|w(z)|}\leq \frac{2r\cos(\pi \theta/2)}{1-r},\quad
    (|z|=r<1).
\end{equation*}
This inequality implies that
$$\frac{1-(1+2\cos(\pi\theta/2))r}{1-r} \leq \mathrm{Re}\,\{F(r)\} \leq \frac{1+(2\cos(\pi\theta/2)-1)r}{1-r}.$$ For functions with positive real part, raising to the power $(\alpha_1+\alpha_2)/2$
 preserves these inequalities, giving the stated bounds. Of course, the constraint $$r \leq \frac{1}{1+2\cos(\pi\theta/2)}$$ ensures the lower bound remains non-negative.
This concludes the proof of Theorem \ref{t21}.
\end{proof}
By substituting $\alpha_1=\alpha_2=\beta$ into Theorem \ref{t21}, we obtain:
\begin{corollary}
Let $f$ be a strongly starlike function of order $\beta$, where $0<\beta\leq 1$. Then
\begin{equation*}
\left(\frac{1-3r}{1-r}\right)^\beta\leq{\rm Re}\left\{\frac{zf'(z)}{f(z)}\right\}\leq
\left(\frac{1+r}{1-r}\right)^\beta,\quad (|z|=r\leq1/3).
\end{equation*}
In particular, if $f$ is a strongly starlike function of order $\beta$, and $|z|=1/3$, then
 \begin{equation*}
0\leq {\rm Re}\left\{\frac{zf'(z)}{f(z)}\right\}\leq 2^\beta.
\end{equation*}
\end{corollary}

\vspace{1cm}
\noindent
\textbf{Author Contributions} R.K., J.S., and H.M. conducted the actual study. R.K., J.S., and H.M. drafted and finalized the manuscript. R.K. prepared the figure for the paper. All authors reviewed the manuscript.\\\\
\noindent
\textbf{Data Availability} Throughout this study, no datasets were generated or analyzed.\\\\
\noindent
\textbf{Competing Interests} The authors declare that they have no conflicting interests.



\end{document}